\documentclass[a4paper,12pt]{article}

\usepackage{mymacros}

\usepackage{booktabs}
\usepackage[normalem]{ulem}

\newcommand{\Gmax}{{G_{\mathrm{max}}}}
\newcommand{\Gmaxbar}{{\Gbar_{\mathrm{max}}}}
\newcommand{\Lmax}{{L_{\mathrm{max}}}}
\newcommand{\Kmax}{{K_{\mathrm{max}}}}

\newcommand{\bXv}{\check{\bX}}
\newcommand{\bYv}{\check{\bY}}
\newcommand{\CF}{\operatorname{CF}}
\renewcommand{\ev}{{\check{e}}}

\title{A note on bimodal singularities and mirror symmetry}
\author{Makiko Mase and Kazushi Ueda}
\date{}
\begin{document}

\maketitle

\begin{abstract}
We discuss the relation between
transposition mirror symmetry of Berlund and H\"{u}bsch
for bimodal singularities and
polar duality of Batyrev for associated toric K3 hypersurfaces.
We also show that
homological mirror symmetry for singularities implies
the geometric construction of Coxeter-Dynkin diagrams
of bimodal singularities
by Ebeling and Ploog.
\end{abstract}

\section{Introduction}

\emph{Mirror symmetry} is a mysterious relationship
between symplectic geometry and complex geometry
motivated by string theory.
It was originally discovered in the context of Calabi-Yau manifolds,
but it has soon become clear that singularity theory
is an important aspect of the theory
through the \emph{Calabi-Yau/Landau-Ginzburg correspondence}
\cite{MR983349,
MR1025421,
MR1104265,
Witten_PN2T2D}.

In this paper,
we discuss the relation between mirror symmetry
for bimodal singularities
and mirror symmetry for K3 surfaces.
Bimodal singularities are a natural class of singularities
which come next to simple (or 0-modal) singularities
and unimodal singularities.
They are classified by Arnold
\cite{MR0420689, MR0467795}
into 8 infinite series and 14 exceptional families.
We will deal only with singularities
defined by invertible polynomials,
since we use \emph{transposition mirror construction}
of Berglund and H\"{u}bsch \cite{Berglund-Hubsch},
which makes sense only for invertible polynomials.
An \emph{invertible polynomial} is a polynomial of the form
\begin{align}
 f(x_1, \ldots, x_n) =  f_A(x_1,\ldots,x_n) = \sum_{i=1}^n x_1^{a_{i1}} \cdots x_n^{a_{in}}
\end{align}
having an isolated singularity at the origin.
Here $A = (a_{ij})_{i,j=1}^n$ is an $n \times n$ matrix
with positive integer entries.
Its \emph{Berglund-H\"{u}bsch transpose}
is the invertible polynomial
\begin{align}
 \fv(x_1,\ldots,x_n)
  = f_{A^T}(x_1,\ldots,x_n)
  = \sum_{i=1}^n x_1^{a_{i1}} \cdots x_n^{a_{ni}}
\end{align}
associated with the transpose matrix $A^T$.
The list of 
bimodal singularities
defined by invertible polynomials in three variables is given
by Ebeling and Ploog \cite[Table 2]{MR3016490}.
In order to relate Berglund-H\"{u}bsch transposition
to mirror symmetry for K3 surfaces,
we will deal only with the cases
where the defining polynomial admits an extension
to an invertible polynomial in four variables
defining a K3 surface in a weighted projective space.
\pref{tb:EP} is obtained
from \cite[Table 4]{MR3016490}
by removing four cases
not satisfying this condition.
\begin{table}[t]
\begin{align*}
\begin{array}{ccccc}
\toprule
\text{Name} & (\qv_1, \qv_2, \qv_3, \qv_4; \hv)  & \Fv(x,y,z,w) &
 F(x,y,z,w) & (q_1, q_2, q_3, q_4; h) \\
\midrule
J_{3,0} & (2,6,9,1;18) & x^6y+y^3+z^2+w^{18} & x^6+xy^3+z^2+w^{18} & (3,5,9,1;18) \\
Z_{1,0} & (2,4,7,1;14) & x^5y+xy^3+z^2+w^{14} & x^5y+xy^3+z^2+w^{14} & (2,4,7,1;14) \\
Q_{2,0} & (2,4,5,1;12) & x^4y+y^3+xz^2+w^{12} & x^4z+xy^3+z^2+w^{12} & (3,7,12,2;24) \\
W_{1,0} & (2,6,3,1;12) & x^6+y^2+yz^2+w^{12} & x^6+y^2z+z^2+w^{12} & (2,3,6,1;12) \\
S_{1,0} & (2,4,3,1;10) & x^5+xy^2+yz^2+w^{10} & x^5y+y^2z+z^2+w^{10} & (3,5,10,2;20) \\
U_{1,0} & (3,3,2,1;9) & x^3+xy^2+yz^3+w^9 & x^3y+y^2z+z^3+w^9 & (2,3,3,1;9) \\
E_{18} & (1,3,4,1;9) & x^5z+y^3+z^2w+w^9 & x^5 + y^3 + xz^2 + z w^9 & (3,5,6,1;15) \\
E_{19} & (1,3,5,1;10) & x^7y+y^3w+z^2+w^{10} & x^7+xy^3+z^2+yw^{10} & (2,4,7,1;14) \\
E_{20} & (1,4,6,1;12) & x^{11}w+y^3+z^2+w^{12} & x^{11}+y^3+z^2+xw^{12} & (6,22,33,5;66) \\
Z_{17} & (1,2,3,1;7) & x^4z+xy^3+z^2w+w^7 & x^4y+y^3+xz^2+zw^7 & (2,4,5,1;12) \\
Z_{19} & (1,3,5,1;10) & x^9w+xy^3+z^2+w^{10} & x^9y+y^3+z^2+xw^{10} & (4,18,27,5;54) \\
Q_{17} & (1,2,3,1;7) & x^5y+y^3w+xz^2+w^7 & x^5z+xy^3+z^2+yw^7 & (1,3,5,1;10) \\
Q_{18} & (1,3,4,1;9) & x^8w+y^3+xz^2+w^9 & x^8z+y^3+z^2+xw^9 & (3,16,24,5;48) \\
W_{18} & (1,4,2,1;8) & x^7w+y^2+yz^2+w^8 & x^7+y^2z+z^2+xw^8 & (4,7,14,3;28) \\
S_{17} & (1,3,2,1;7) & x^6w+xy^2+yz^2+w^7 & x^6y+y^2z+z^2+xw^7 & (1,2,4,1;8) \\
U_{16} & (1,2,2,1;6) & x^5w+y^2z+yz^2+w^6 & x^5+y^2z+yz^2+xw^6 & (3,5,5,2;15) \\
\bottomrule
\end{array}
\end{align*}
\caption{Transposition of invertible polynomials}
\label{tb:EP}
\end{table}
Here, the invertible polynomial
$F(x,y,z,w)$ is a four-variable extension
of an invertible polynomial $f(x,y,z)$ in three variables,
whose Berglund-H\"{u}bsch transpose $\fv(x,y,z)$
defines a bimodal singularity.
Note that the Berglund-H\"{u}bsch transpose
$\Fv(x, y, z, w)$ of $F(x,y,z,w)$
may not be a four-variable extension
of $\fv(x,y,z)$ in general;
\begin{align}
 \Fv(x,y,z,0) \ne \fv(x,y,z).
\end{align}
The weight $(q_1, q_2, q_3, q_4)$ is the primitive weight
which makes $F$ into a weighted homogeneous polynomial
of degree $h$.
The weight $(\qv_1, \qv_2, \qv_3, \qv_4)$ is defined similarly
as the primitive weight
which makes $\Fv$ weighted homogeneous of degree $\hv$.

The \emph{group of maximal diagonal symmetries} of $F$
is defined by
\begin{align}
 \Gmax(F)
  &:= \lc \diag(\alpha, \beta, \gamma, \delta) \in \GL_4(\bC)
   \relmid F(\alpha x, \beta y, \gamma z, \delta w) = F(x, y, z, w) \rc.
\end{align}
Since any element of $\Gmax(F)$ is contained
in the subgroup $(\bQ/\bZ)^4 \subset (\bCx)^4$,
we can use an additive notation and
write $\frac{1}{n}(a,b,c,d)$ for
$(\zeta_n^a, \zeta_n^b, \zeta_n^c, \zeta_n^d)$,
where $\zeta_n = \exp(2 \pi \sqrt{-1}/n)$.
In the additive notation,
the group $\Gmax(F)$ of maximal diagonal symmetries
is generated by the column vectors
\begin{align*}
 \rho_i = \lb (A^{-1})_{1i}, (A^{-1})_{2i}, (A^{-1})_{3i}, (A^{-1})_{4i} \rb,
  \qquad i = 1, 2, 3, 4,
\end{align*}
of the inverse matrix of $A$.

For a subgroup $G$ of $\Gmax(F)$ containing
$
 J
  = J(F)
  = \frac{1}{h}(q_1,q_2,q_3,q_4),
$
the natural action of $G$ on $\bC^4$ induces
an action of $\Gbar = G / \la J \ra$ on
the weighted projective space
$\bP=\bP(q_1,q_2,q_3,q_4)$.
This gives an action of $\Gbar$
on the hypersurface
$$
 Y := \lc [x:y:z:w] \in \bP \relmid F(x,y,z,w) = 0 \rc,
$$
since $F(x,y,z,w)$ is invariant under the action of $\Gbar$.
The quotient stack $\bX=[\bP / \Gbar]$ is a toric stack,
and we write the group of characters of its dense torus as $M$.
Since $F$ has an isolated critical point at the origin,
the hypersurface $\bY$ of $\bX$
defined by $F$ is a smooth Deligne-Mumford stack.
We write the Newton polytope of $F$
in $M_\bR = M \otimes_\bZ \bR$ as $\Delta_{(F,G)}$.

The \emph{Berglund-H\"{u}bsch transpose} of $G$ is defined by
\begin{align*}
 \Gv :=
 \lc \rhov_1^{r_1} \rhov_2^{r_2} \rhov_3^{r_3} \rhov_4^{r_4}
   \in \Gmax(\Fv) \relmid
    \begin{pmatrix} r_1 & r_2 & r_3 & r_4 \end{pmatrix}
    A^{-1}
    \begin{pmatrix} a_1 \\ a_2 \\ a_3 \\ a_4 \end{pmatrix}
   \in \bZ
   \text{ for all }
    \rho_1^{a_1} \rho_2^{a_2} \rho_3^{a_3} \rho_4^{a_4} \in G
   \rc.
\end{align*}
The Newton polytope $\Delta_{(\Fv,\Gv)}$ of $\Fv$
in $\Mv_\bR$ is defined similarly,
where $\Mv$ is the group of characters of the dense torus
in the toric stack
$
 \bXv=[\bP(\qv_1,\qv_2,\qv_3,\qv_4) / \Gv].
$

Let $N = \Hom(M, \bZ)$ be the dual lattice of $M$.
The \emph{polar dual} of a polytope $\Delta$ in $M_\bR$
is defined by
\begin{align}
 \Delta^\circ =
  \lc n \in N_\bR \relmid \la n, m \ra \ge -1 \text{ for any } m \in \Delta \rc.
\end{align}
A lattice polytope is \emph{reflexive}
if the polar dual polytope is a lattice polytope.
Polar duality of reflexive polytope is an essential ingredient
of mirror construction by Batyrev
\cite{Batyrev_DPMS}.
%
%
In this paper, we give a relation
between transposition and polar duality
for bimodal singularities.

\begin{theorem} \label{th:main1}
For any polynomial $F$ in \pref{tb:EP},
there is an isomorphism
$\varphi : N \simto \Mv$
of abelian groups
and a reflexive polytope
$\Delta \subset M_\bR$
such that
\begin{align}
 \Delta_{(F,G)} \subset \Delta
 \quad \text{and} \quad
 \Delta_{(\Fv,\Gv)} \subset \varphi(\Delta^\circ),
\end{align}
where $G$ is the unique lift
of $\Gmax(f)$
to a subgroup of $\Gmax(F) \cap \SL_4(\bC)$.
\end{theorem}

This shows that
for K3 surfaces associated with bimodal singularities,
Berglund-H\"{u}bsch mirror construction
can be regarded as a special case
of Batyrev mirror construction.
This generalizes a related result
for unimodal singularities
by Kobayashi \cite[Thereom 4.3.9]{Kobayashi_DW}.

The group $\Gbar$,
which is trivial for unimodal singularities,
comes from \emph{homological mirror symemtry}
\cite{Kontsevich_HAMS} for singularities.

\begin{conjecture} \label{cj:hms}
For any invertible polynomial $f$,
there is an equivalence
\begin{align} \label{eq:hms}
 D^b \Fuk \fv \cong D^b_\sing (\gr^\Lmax R)
\end{align}
of triangulated categories.
\end{conjecture}

The category $\Fuk \fv$ on the left hand side
is the \emph{Fukaya-Seidel category},
whose objects are vanishing cycles of $\fv$
and whose spaces of morphisms are
Lagrangian intersection Floer complexes
\cite{Seidel_PL}.
It is a categorification of the Milnor lattice,
which is an important invariant in singularity theory.
The category $D^b_\sing(\gr^\Lmax R)$
on the right hand side
is the \emph{stable derived category},
also known as the \emph{singularity category},
defined as the quotient 
\begin{align}
 D^b_\sing(\gr^\Lmax R) = D^b(\gr^\Lmax R) / D^\perf(\gr^\Lmax R)
\end{align}
of the bounded derived category $D^b(\gr^\Lmax R)$
of finitely-generated $\Lmax$-graded $R$-modules
by the full subcategory $D^\perf(\gr^\Lmax R)$
consisting of bounded complexes of projectives.
The ring
\begin{align} \label{eq:R}
 R = \bC[x_1, \ldots, x_n] / (f(x_1,\ldots,x_n))
\end{align}
is the coordinate ring of zero of $f$,
which is graded by the abelian group
\begin{align} \label{eq:L}
 \Lmax = \bZ \vecx_1 \oplus \cdots \oplus \bZ \vecx_n \oplus \bZ \vecc
  / \lb a_{i1} \vecx_1 + \cdots + a_{in} \vecx_n - \vecc \rb_{i=1}^n
\end{align}
of rank one.
Stable derived categories
are first introduced
by Buchweitz \cite{Buchweitz_MCM}
motivated by a work of Eisenbud \cite{Eisenbud_HACI},
and later by Orlov \cite{Orlov_TCS}
following an idea of Kontsevich.

In the case of bimodal singularities,
Ebeling and Ploog gave the following geometric construction
of Coxeter-Dynkin diagrams:

\begin{theorem}[{\cite[Theorem 7]{MR3016490}}]
 \label{th:EP}
For any bimodal singularity $\fv$
in \cite[Table 2]{MR3016490},
there is a distinguished basis $(\Delta_i)_{i=1}^\mu$
of vanishing cycles
and a collection $(\cE_i)_{i=1}^\mu$
of objects of $D^b \coh \bY$ such that
\begin{align}
 \Delta_i \cdot \Delta_j
  = - \sum_k (-1)^k \dim \Ext^k(\cE_i, \cE_j).
\end{align}
\end{theorem}

Here $\Delta_i \cdot \Delta_j$ is the intersection number
of vanishing cycles.
\pref{th:main2} below shows that
\pref{th:EP} is an evidence to \pref{cj:hms}:

\begin{theorem} \label{th:main2}
\pref{cj:hms} implies \pref{th:EP}.
\end{theorem}

\pref{cj:hms} is known
for Sebastiani-Thom sum of polynomials
of types A and D
\cite{Futaki-Ueda_BP,Futaki-Ueda_Dn}.
Since Fukaya-Seidel categories are invariant
under suspension
\cite{Seidel_suspension},
\pref{cj:hms} for a suspension
of a polynomial in two variables
is reduced to the case of curve singularities,
which is announced by Takahashi
\cite{Takahashi_curve}.
Together, they cover 10 out of 20 cases
in \cite[Table 2]{MR3016490},
and \pref{th:main2} gives an alternative proof
of \pref{th:EP} for such cases.

\emph{Acknowledgment}:
K.~U. is supported by JSPS Grant-in-Aid for Young Scientists No.~24740043.
A part of this work is done
while K.~U. is visiting Korea Institute for Advanced Study,
whose hospitality and wonderful working environment
is gratefully acknowledged.

\section{Transposition and polar duality}
 \label{sc:tpd}

Let
$
 F = \sum_{i=1}^n x_1^{a_{i1}} \cdots x_n^{a_{in}}
$
be an invertible polynomial
associated with a matrix $A = (a_{ij})_{i,j=1}^n$.
The group $\Lmax$ defined in \eqref{eq:L}
is the group of characters of
\begin{align} \label{eq:K}
 \Kmax = \lc (\alpha_1, \ldots, \alpha_n) \in (\bCx)^n \relmid
  \alpha_1^{a_{11}} \cdots \alpha_n^{a_{1n}}
   = \cdots
   = \alpha_1^{a_{n1}} \cdots \alpha_n^{a_{nn}} \rc.
\end{align}
The group
\begin{align}
 \Gmax(F) = \lc \diag(\alpha_1, \ldots, \alpha_n) \in \GL_n(\bC) \relmid
  \prod_{j=1}^n \alpha_j^{a_{ij}} = 1 \text{ for any }  i = 1, \ldots, n \rc
\end{align}
of maximal diagonal symmetries of $F$
is the kernel of the map
\begin{align}
\begin{array}{ccc}
 \Kmax & \to & \bCx \\
 \vin & & \vin \\
 (\alpha_1, \ldots, \alpha_n) & \mapsto & \alpha_1^{a_{11}} \cdots \alpha_n^{a_{1n}},
\end{array}
\end{align}
so that one has an exact sequence
\begin{align} \label{eq:exact1}
 1 \to \Gmax(F) \to \Kmax \to \bCx \to 1.
\end{align}
Let $(q_1, \ldots, q_n)$ be the primitive weight
which makes $F$ weighted homogeneous of degree $h$.
Set
$
 J(F)
  = \diag(\zeta_h^{q_1}, \ldots, \zeta_h^{q_n})
  \in \Gmax(F)
$
where $\zeta_h = \exp(2 \pi \sqrt{-1}/h)$
is the primitive $h$-th root of unity.
Then one has an exact sequence
\begin{align} \label{eq:exact2}
 1 \to \bCx \xto{\varphi} \Kmax \to \Gmaxbar(F) \to 1
\end{align}
where
$
 \varphi(\alpha) = (\alpha^{q_1}, \ldots, \alpha^{q_n})
$
and
$\Gmaxbar(F) = \Gmax(F) / \la J(F) \ra$.

Let $G$ be a subgroup of $\Gmax(F)$
containing $J(F)$ and set
$
 \Gbar = G / \la J(F) \ra.
$
The inverse image of $\Gbar$ in \eqref{eq:exact2}
will be denoted by $K$:
\begin{align} \label{eq:exact3}
 1 \to \bCx \xto{\varphi} K \to \Gbar \to 1.
\end{align}
The group of characters of $K$ will be denoted by $L$.
Let $\bT$ be the cokernel of the inclusion
$K \subset (\bCx)^n$,
which is isomorphic to $(\bCx)^{n-1}$.
Let $M = \Hom(\bT, \bCx)$
be the group of characters of $\bT$.
One has an exact sequence
\begin{align} \label{eq:exact4}
 1 \to K \to (\bCx)^n \to \bT \to 1
\end{align}
of abelian groups,
which induces
an exact sequence
\begin{align}
 1 \to M \to \bZ^n \to L \to 1
\end{align}
of groups of characters.
The natural action of $(\bCx)^n$
on $\bC^n = \Spec U$
restricts to an action of $K$,
which induces an $L$-grading
on the coordinate ring $U = \bC[x_1,\ldots,x_n]$.
The quotient stack
$
 \bX = [(\bC^n \setminus \bszero)/K]
$
is a toric Deligne-Mumford stack
with the dense torus $\bT$
and the Picard group $L$.

The polynomial $F$ gives a section
of a line bundle on $\bX$,
so that its Newton polytope
\begin{align}
 \Delta_F
  = \Conv \lc (a_{i1}, \ldots, a_{in}) \rc_{i=1}^n
  \subset \bR^n
\end{align}
is a lattice polytope
in the translate of $M_\bR = M \otimes \bR$.
This polytope,
considered as a lattice polytope in $M_\bR$,
will be denoted by $\Delta_{(F,G)}$.
It is well-defined up to translation by a lattice vector.
Let $S = U / (F)$ be the quotient ring of $U$
by the ideal generated by the polynomial $F$.
Since $F$ has an isolated critical point at the origin,
the hypersurface
$
 \bY = [(\Spec S \setminus \bszero)/K] \subset \bX
$
is a smooth Deligne-Mumford stack.

The Berglund-H\"{u}bsch transpose
of the pair $(F, G)$ is defined
in \cite{Berglund-Hubsch,Krawitz}
as the pair $(\Fv, \Gv)$ of the invertible polynomial
$
 \Fv = \sum_{i=1}^n x_1^{a_{1i}} \cdots x_n^{a_{ni}}
$
and the group
\begin{align*}
 \Gv :=
 \lc \rhov_1^{r_1} \cdots \rhov_n^{r_n}
   \in \Gmax(\Fv) \relmid
    \begin{pmatrix} r_1 & \cdots & r_n \end{pmatrix}
    A^{-1}
    \begin{pmatrix} a_1 \\ \vdots \\ a_n \end{pmatrix}
   \in \bZ
   \text{ for all }
    \rho_1^{a_1} \cdots \rho_n^{a_n} \in G
   \rc.
\end{align*}
One can easily see
that
\begin{itemize}
 \item
the group $\Gv$ contains $J(\Fv)$
if and only if the group $G$ is contained in $\SL_n(\bC)$, and
 \item
the group $\Gv$ is contained in $\SL_n(\bC)$
if and only if the group $G$ contains $J(F)$.
\end{itemize}
Assume that $G$ (and hence $\Gv$)
contains $J(F)$ and is contained in $SL_n(\bC)$.
The toric stack $\bXv$,
its hypersurface $\bYv$,
the dense torus $\bTv \subset \bXv$,
the lattice $\Mv = \Hom(\bXv, \bCx)$
and the lattice polytope $\Delta_{(\Fv,\Gv)} \subset \Mv_\bR$
are defined in the same way as $(F,G)$.

For a lattice polytope $\Delta$ in $M_\bR$,
its {\em polar dual polytope} $\Delta^\circ$ is defined by
\begin{align*}
 \Delta^\circ = \lc n \in N_\bR \relmid
  \la n, m \ra \ge -1 \text{ for any } m \in \Delta \rc,
\end{align*}
where $N = \Hom(M, \bZ)$ and $N_\bR = N \otimes_\bZ \bR$.
A lattice polytope is \emph{reflexive}
if its polar dual polytope is a lattice polytope.
According to Batyrev
\cite{Batyrev_DPMS},
the mirror of an anti-canonical hypersurface
of a toric weak Fano manifold
associated with a reflexive polytope
is an anti-canonical hypersurface
of a toric weak Fano manifold
associated with the polar dual polytope.
The following problem asks the relation
between transposition and polar duality:

\begin{problem} \label{pb:main1}
There exist a lattice isomorphism
$
 \varphi : N \simto \Mv
$
and a reflexive polytope $\Delta \subset M$
such that $\Delta_{(F,G)} \subset \Delta$ and
$\Delta_{(\Fv,\Gv)} \subset \varphi(\Delta^\circ)$.
\end{problem}

\pref{th:main1} gives an affirmative answer
to \pref{pb:main1}
for invertible polynomials
in \pref{tb:EP}
with a specific group $G$
coming from mirror symmetry for singularities.

\section{Homological mirror symmetry for singularities}
 \label{sc:hms}

Let
$
 f = \sum_{i=1}^n x_1^{a_{i1}} \cdots x_n^{a_{in}}
$
be an invertible polynomial and
$
 \fv
$
be its Berglund-H\"{u}bsch transpose.
The coordinate ring
$
 R = \bC[x_1,\ldots,x_n]/(f)
$
is graded by the abelian group
$
 L = \Lmax(f)
$
of rank one defined in \eqref{eq:L}.
By perturbing $\fv$ slightly,
one obtains an exact symplectic Lefschetz fibration
\cite[Section 6]{Seidel_suspension},
which we write as $\fv$ again by abuse of notation.
Choose a distinguished basis
$(\Delta_i)_{i=1}^\mu$ of vanishing cycles,
where $\mu$ is the Milnor number of $\fv$.
The \emph{Fukaya category of Lefschetz fibration},
also known as the \emph{Fukaya-Seidel category},
is an $A_\infty$-category
such that
\begin{itemize}
 \item
the set of objects is $(\Delta_i)_{i=1}^\mu$, and
 \item
the space of morphisms are given by
\begin{align}
 \hom^*(\Delta_i, \Delta_j) =
  \begin{cases}
   0 & i > j, \\
   \bC \cdot \id_{\Delta_i} & i = j, \\
   \CF^*(\Delta_i, \Delta_j) & i < j,
  \end{cases}
\end{align}
where $\CF^*(\Delta_i, \Delta_j)$ is 
the Lagrangian intersection Floer complex, and
 \item
the composition is defined by virtual couting
of holomorphic disks with Lagrangian boundary conditions.
\end{itemize}

We refer the reader to \cite{Seidel_PL}
for the definition of Fukaya-Seidel categories.
The quasi-equivalence class
of the derived Fukaya-Seidel category $D^b \Fuk \fv$
is independ of various auxiliary choices
in the definition of $D^b \Fuk \fv$,
such as the choices of a perturbation
or a distinguished basis of vanishing cycles.
The Grothendieck group $K(\Fuk \fv)$
of the Fukaya-Seidel category
is freely generated
by the distinguished basis of vanishing cycles.
The Euler number of the Floer complex
is equal to minus the intersection number
of Lagrangian submanifolds;
\begin{align}
 \sum_{k} (-1)^k \dim \CF^k(\Delta_i, \Delta_j)
  = - \Delta_i \cdot \Delta_j.
\end{align}
The lattice of vanishing cycles
equipped with the intersection form
is called the \emph{Milnor lattice}.
The Euler form on the Grothendieck group
of a triangulated category is defined by
\begin{align}
 \chi(\cE, \cF)
  = \sum_k (-1)^k \dim \Hom^k(\cE, \cF).
\end{align}
The Fukaya-Seidel category is a
\emph{categorification}
of the Milnor lattice,
in the sense that the Grothendieck group
equipped with the symmetrized Euler form
\begin{align} \label{eq:1}
 (\Delta_i, \Delta_j)
  = \chi(\Delta_i, \Delta_j) + (-1)^{n-1} \chi(\Delta_j, \Delta_i)
\end{align}
is isomorphic to the Milnor lattice up to sign;
\begin{align} \label{eq:2}
 (\Delta_i, \Delta_j) = - \Delta_i \cdot \Delta_j.
\end{align}

Under \pref{cj:hms},
the graded stable derived category
$D^b_\sing (\gr^L R)$ is also a categorification
of the Milnor lattice of $\fv$,
where $R$ is the coordinate ring
$\bC[x_1,\ldots,x_n]/(f)$
of the singularity
defined by the transpose of $\fv$,
and $L=\Lmax(f)$ is the maximal grading group for $f$.

Let $F(x_1,\ldots,x_{n+1}) \in \bC[x_1,\ldots,x_{n+1}]$ be
a polynomial in $n+1$ variables
satisfying the following assumption:
\begin{assumption} \label{as:F}
\ 
\begin{itemize}
 \item
$F(x_1,\ldots,x_n,0)=f(x_1,\ldots,x_n)$.
 \item
$F$ is weighted homogeneous of degree
$(q_1,\ldots,q_{n+1};h)$
satisfying $q_1+\cdots+q_{n+1}=h$.
 \item
$F$ has an isolated critical point at the origin.
 \item
For any $g = \diag(\alpha_1,\ldots,\alpha_n) \in \Gmax(f)$,
the polynomial $F$ is invariant under the action
of its lift $\gtilde = \diag(\alpha_1,\ldots,\alpha_n,
\alpha_1^{-1} \cdots \alpha_n^{-1}) \in \SL_{n+1}(\bC)$.
\end{itemize}
\end{assumption}
This condition is satisfied
for all polynomials in \cite[Table 2]{MR3016490}.
We write the image of the injective map
$
 \Gmax(f) \ni g
  \mapsto \gtilde \in \Gmax(F) \cap \SL_{n+1}(\bC)
$
as $G$.
The group $G$ defines the group $K$
fitting in the exact sequence \eqref{eq:exact3}.
Let $S=\bC[x_1,\ldots,x_n]/(F)$ be the coordinate ring
of the zero of $F$,
which is graded by the group $L$
of characters of $K$.
The corresponding projective stack
will be denoted by
$\bY = \bProj S = [(\Spec S \setminus \bszero)/K]$.

We have the following two theorems:

\begin{theorem}[{\cite[Theorem 2.13]{Orlov_DCCSTCS}}]
 \label{th:Orlov}
There is an equivalence
\begin{align} \label{eq:Orlov}
 \Psi : D^b_\sing(\gr^L S)
  \to D^b \coh \bY
\end{align}
of triangulated categories.
\end{theorem}


\begin{theorem}[{\cite[Theorem 1.1]{Ueda_HSSDC}}]
There is a functor
\begin{align}
 \Phi : D^b_\sing(\gr^L R) \to D^b_\sing(\gr^L S)
\end{align}
such that
\begin{align} \label{eq:Phi}
 \Hom^i(\Phi(\cE), \Phi(\cF))
  &\cong \Hom^i(\cE, \cF) \oplus \Hom^{n-i-1}(\cF, \cE).
\end{align}
\end{theorem}

Although both theorems are stated for $\bZ$-graded rings
in the references,
the generalization to the $L$-graded case is straighforward.
\pref{th:f} below contains \pref{th:main2}
as a special case:

\begin{theorem} \label{th:f}
Let $f$ be an invertible polynomial in $n$ variables
satisfying \pref{cj:hms}, and
$F$ be a polynomial in $n+1$ variables
satisfying \pref{as:F}.
Then for any distinguished basis
$(\Delta_i)_{i=1}^\mu$
of vanishing cycles of $\fv$,
there is a collection $(\cE_i)_{i=1}^\mu$
of objects of $D^b \coh \bY$
satisfying
\begin{align} \label{eq:E}
 \Delta_i \cdot \Delta_j
  = - \chi(\cE_i, \cE_j).
\end{align}
\end{theorem}

\begin{proof}
The distinguished basis $(\Delta_i)_{i=1}^\mu$
of vanishing cycles are objects
of the Fukaya-Seidel category.
Let $(\cF_i)_{i=1}^\mu$ be their images
in $D^b_\sing(\gr^L R)$
under the equivalence \eqref{eq:hms},
so that
\begin{align} \label{eq:F}
 \chi(\Delta_i, \Delta_j)
  &= \chi(\cF_i, \cF_j).
\end{align}
If we set
$
 \cG_i = \Phi(\cF_i)
$
and
$
 \cE_i = \Psi(\cG_i)
$
for $i=1, \ldots, \mu$,
then we have
\begin{align} \label{eq:3}
 \chi(\cG_i, \cG_j)
  &= \chi(\cF_i, \cF_j) + (-1)^{n-1} \chi(\cF_j, \cF_i)
\end{align}
by \eqref{eq:Phi} and
\begin{align} \label{eq:4}
 \chi(\cE_i, \cE_j)
  &= \chi(\cG_i, \cG_j)
\end{align}
by \pref{th:Orlov}.
Now \eqref{eq:E} follows from \eqref{eq:1},
\eqref{eq:2}, \eqref{eq:3} and \eqref{eq:4}.
\end{proof}

\section{Proof of \pref{th:main1}}
 \label{sc:main1}


\subsection{$J_{3,0}$-singularity}

The invertible polynomial
\begin{align*} \label{eq:j30}
 \fv = x^6 y + y^3 + z^2
\end{align*}
defines a $J_{3,0}$-singularity.
Its Berglund-H\"{u}bsch transpose
\begin{align*}
 f = x^6 + x y^3 + z^2
\end{align*}
is weighted homogeneous of degree 
$
 (q_1, q_2, q_3; h) = (3, 5, 9; 18).
$
The group of maximal diagonal symmetries
is given by
\begin{align*}
 \Gmax(f)
  &= \{ (\alpha, \beta, \gamma) \in (\bCx)^3 \mid
   \alpha^6 = \alpha \beta^3 = \gamma^2 = 1 \}
  = \la J(f), \, \frac{1}{2}(0, 0, 1) \ra
\end{align*}
where
$
 J(f) = \frac{1}{18}(3,5,9).
$
The invertible polynomial
\begin{align*}
 F = x^6 + x y^3 + z^2 + w^{18}
\end{align*}
is a four-variable extension of $\fv$
satisfying \pref{as:F}.
The subgroup $G$ of $\Gmax(F)$
is given by
\begin{align*}
 G = \la J(F), \, \frac{1}{2}(0,0,1,1) \ra
\end{align*}
where
$
 J(F)
  = \frac{1}{18}(3, 5, 9, 1).
$
The lattice $M$ is given by
\begin{align*}
 M = \lc (i, j, k, l) \in \bZ^4 \relmid
  3 i + 5 j + 9 k + l = 0 \text{ and } k+l \equiv 0 \text{ mod } 2 \rc.
\end{align*}
For each monomial
$
 x^{a_{i1}} y^{a_{i2}} z^{a_{i3}} w^{a_{i4}}
$
in $F$, the Laurent monomial
$
 x^{a_{i1}} y^{a_{i2}} z^{a_{i3}} w^{a_{i4}}/xyzw
$
defines a regular function on $\bT$,
and hence gives a lattice point
$(a_{i1}-1, a_{i2}-1, a_{i3}-1, a_{i4}-1)$
in $M$.
By identifying $M$ with $\bZ^3$
by the basis
\begin{align*}
 e_1 &= (5,-1,-1,-1), &
 e_2 &= (0,2,-1,-1), &
 e_3 &= (-1,-1,1,-1),
\end{align*}
these points are given by
\begin{align*}
\begin{array}{c|c|c|c}
 x^6 & x y^3 & z^2 & w^{18} \\
 \hline
 (1,0,0) & (0,1,0) & (0,0,1) & (-2,-6,-9)
\end{array}
.
\end{align*}
The Newton polytope $\Delta_{(F,G)}$ is
the convex hull of these points.
Similarly,
we choose the basis of the group
$$
 \Mv := \lc (i, j, k, l) \in \bZ^4 \relmid 2 i + 6 j + 9 k + l = 0 \rc
$$
of characters
of the quotient of the dense torus of
$\bXv = \bP(2, 6, 9, 1)$
as
\begin{align*}
 \ev_1 &= (1,0,0,-2), &
 \ev_2 &= (0,1,0,-6), &
 \ev_3 &= (0,0,1,-9),
\end{align*}
so that the Newton polytope $\Delta_{(\Fv, \Gv)}$
of $\Fv/xyzw$ is the convex hull of
\begin{align*}
\begin{array}{c|c|c|c}
 x^6y & y^3 & z^2 & w^{18} \\
 \hline
 (5,0,-1) & (-1,2,-1) & (-1,-1,1) & (-1,-1,-1)
\end{array}
.
\end{align*}
Instead of $\Delta_{(F,G)}$ which is not reflexive,
take the polytope
\begin{align*}
 \Delta := \lc
  (1,0,0),(0,1,0),(0,0,1),(-2,-6,-9),(-1,-2,-4)
 \rc,
\end{align*}
which clearly contains $\Delta_{(F,G)}$.
The polytope $\Delta$ is reflexive,
whose polar dual polytope is given by
\begin{align*}
 \Delta^\circ := \lc
  (5,0,-1),(-1,2,-1),(-1,-1,1),(-1,-1,-1),(7,-1,-1)
 \rc.
\end{align*}
$\Delta^\circ$ clearly contains $\Delta_{(\Fv,\Gv)}$,
so that \pref{th:main1} holds for this case.

\subsection{$Z_{1,0}$-singularity}

The invertible polynomial
$
 f = x^5 y + x y^3 + z^2
$
is weighted homogeneous of degree 
$
 (q_1, q_2, q_3; h) = (2, 4, 7; 14).
$
The group of maximal diagonal symmetries is given by
\[
 \Gmax(f)
  = \la J(f), \frac{1}{2}(1,1,0) \ra.
\]
The four-variable extension
$
 F = x^5 y + x y^3 + z^2 + w^{14}
$
of $f$ is weighted homogeneous
of weight $(2,4,7,1)$ and satisfies \pref{as:F}.
The subgroup $G$ of $\Gmax(F)$
is given by
$
 G = \la J(F), \frac{1}{2}(1,1,0,0) \ra.
$
The lattice $M$ is given by
$$
 M = \lc (i, j, k, l) \in \bZ^4 \relmid
  2i+4j+7k+l=0 \text{ and } i+j \equiv 0 \text{ mod } 2 \rc.
$$
We choose a basis of $M$ as
\begin{align*}
 e_1 &= (1,1,0,-6), &
 e_2 &= (1,-1,0,2), &
 e_3 &= (0,0,1,-7).
\end{align*}
With respect this basis,
the polytope $\Delta_{(F,G)}$ is the convex hull of
\begin{align*}
\begin{array}{c|c|c|c}
 x^5 y & x y^3 & z^2 & w^{14} \\
 \hline
 (2,2,-1) & (1,-1,-1) & (-1,0,1) & (-1,0,-1)
\end{array}
.
\end{align*}
Similarly,
we choose the basis of the group
$$
 \Mv := \lc (i, j, k, l) \in \bZ^4 \relmid 2 i + 4 j + 7 k + l = 0 \rc
$$
as
\begin{align*}
 \ev_1 &= (2,1,-1,-1), &
 \ev_2 &= (2,-1,0,0), &
 \ev_3 &= (-1,-1,1,-1),
\end{align*}
so that the Newton polytope $\Delta_{(\Fv, \Gv)}$
of $\Fv/xyzw$ is the convex hull of
\begin{align*}
\begin{array}{c|c|c|c}
 x^5y & xy^3 & z^2 & w^{14} \\
 \hline
 (1,1,0) & (1,-1,0) & (0,0,1) & (-6,2,-7)
\end{array}
.
\end{align*}
The polytope
\begin{align*}
 \Delta := \lc
  (2,2,-1),(1,-1,-1),(-1,0,1),(-1,0,-1),(2,3,-1),(0,-1,-1)
 \rc
\end{align*}
contains $\Delta_{(F,G)}$,
and its polar dual polytope
\begin{align*}
 \Delta^\circ := \lc
  (1,1,0),(1,-1,0),(0,0,1),(-6,2,-7),(0,2,-1),(-2,0,-3)
 \rc
\end{align*}
contains $\Delta_{(\Fv,\Gv)}$.

\subsection{$Q_{2,0}$-singularity}

The invertible polynomial
$
 f = x^4z+xy^3+z^2
$
is weighted homogeneous of degree 
$
 (q_1, q_2, q_3; h) = (3,7,12;24).
$
The group of maximal diagonal symmetries is given by
\[
 \Gmax(f)
  = \la J(f) \ra.
\]
The four-variable extension
$
 F = x^4z+xy^3+z^2+w^{12}
$
of $f$ is weighted homogeneous
of weight $(3,7,12,2)$ and satisfies \pref{as:F}.
The subgroup $G$ of $\Gmax(F)$
is given by
$
 G = \la J(F)\ra.
$
The lattice $M$ is given by
$$
 M = \lc (i, j, k, l) \in \bZ^4 \relmid 3i + 7j + 12k + 2l = 0 \rc.
$$
We choose a basis of $M$ as
\begin{align*}
 e_1 &= (0,0,-1,6), & e_2 &= (2,0,-1,3), & e_3 &= (1,-1,0,2). 
\end{align*}
With respect this basis,
the polytope $\Delta_{(F,G)}$ is the convex hull of
\begin{align*}
\begin{array}{c|c|c|c}
 z^2 & w^{12}& xy^3& x^4z \\  \hline (0,-1,1) & (2,-1,1) & (0,1,-2) & (-1,1,1)
\end{array}
.
\end{align*}
Similarly,
we choose the basis of the group
$$
 \Mv := \lc (i, j, k, l) \in \bZ^4 \relmid  2i + 4 j + 5 k + l = 0 \rc
$$
as
\begin{align*}
\ev_1 &= (-1,0,0,2), & \ev_2 &= (1,1,-1,-1), & \ev_3 &= (1,-2,1,1),
\end{align*}
so that the Newton polytope $\Delta_{(\Fv, \Gv)}$
of $\Fv/xyzw$ is the convex hull of
\begin{align*}
\begin{array}{c|c|c|c}
xz^2 & w^7z& x^4y & y^3  \\ \hline (-1,-1,0) & (3,1,1) & (0,2,1) & (0,0,-1)
\end{array}
.
\end{align*}
The polytope
\begin{align*}
 \Delta := \lc (0,1,-2),(0,1,1),(-1,1,1),(0,-1,1),(2,-1,1) \rc
\end{align*}
contains $\Delta_{(F,G)}$,
and its polar dual polytope
\begin{align*}
 \Delta^\circ := \lc (-1,-1,0),(0,-1,0),(0,0,-1),(6,3,2),(0,3,2) \rc
\end{align*}
contains $\Delta_{(\Fv,\Gv)}$.

\subsection{$W_{1,0}$-singularity}

The invertible polynomial
$
 f = x^6+y^2z+z^2
$
is weighted homogeneous of degree 
$
 (q_1, q_2, q_3; h) = (2,3,6;12).
$
The group of maximal diagonal symmetries is given by
\[
 \Gmax(f)
  = \la J(f), \, \frac{1}{2}(1,0,0) \ra.
\]
The four-variable extension
$
 F = x^6+y^2z+z^2+w^{12}
$
of $f$ is weighted homogeneous
of weight $(2,3,6,1)$ and satisfies \pref{as:F}.
The subgroup $G$ of $\Gmax(F)$
is given by
$
 G = \la J(F), \, \frac{1}{2}(1,0,0,1) \ra.
$
The lattice $M$ is given by
$$
 M = \lc (i, j, k, l) \in \bZ^4 \relmid 2i + 3 j + 6 k + l = 0 \text{ and } i+l \equiv 0 \text{ mod 2}  \rc.
$$
We choose a basis of $M$ as
\begin{align*}
 e_1 &= (1,1,0,-5), & e_2 &= (1,-1,0,1), & e_3 &= (0,0,1,-6). 
\end{align*}
With respect this basis,
the polytope $\Delta_{(F,G)}$ is the convex hull of
\begin{align*}
\begin{array}{c|c|c|c}
 z^2 &w^{12} & x^6&y^2z  \\  \hline (-1,0,1) & (-1,0,-1) & (2,3,-1) & (0,-1,0)
\end{array}
.
\end{align*}
Similarly,
we choose the basis of the group
$$
 \Mv := \lc (i, j, k, l) \in \bZ^4 \relmid  2i + 3 j + 6 k + l = 0 \rc
$$
as
\begin{align*}
\ev_1 &= (2,-1,0,-1), & \ev_2 &= (3,0,-1,0), & \ev_3 &= (-1,1,0,-1),
\end{align*}
so that the Newton polytope $\Delta_{(\Fv, \Gv)}$
of $\Fv/xyzw$ is the convex hull of
\begin{align*}
\begin{array}{c|c|c|c}
z^2 &w^{12} &x^6 &y^2z  \\ \hline (1,-1,0) & (-5,1,-6) & (1,1,0) & (0,0,1)
\end{array}
.
\end{align*}
If we set
\begin{align*}
 \Delta := \Delta_{(F,G)} = \lc (-1,0,1), (-1,0,-1), (2,3,-1), (0,-1,0) \rc,
\end{align*}
then its polar dual polytope
\begin{align*}
 \Delta^\circ := \lc (1,1,0),(1,-1,0),(-5,1,-6),(-1,1,2) \rc
\end{align*}
contains $\Delta_{(\Fv,\Gv)}$.

\subsection{$S_{1,0}$-singularity}

The invertible polynomial
$
 f = x^5y+y^2z+z^2
$
is weighted homogeneous of degree 
$
 (q_1, q_2, q_3; h) = (3,5,10;20).
$
The group of maximal diagonal symmetries is given by
\[
 \Gmax(f)
  = \la J(f) \ra.
\]
The four-variable extension
$
 F = x^5y+y^2z+z^2+w^{10}
$
of $f$ is weighted homogeneous
of weight $(3,5 ,10 ,2 )$ and satisfies \pref{as:F}.
The subgroup $G$ of $\Gmax(F)$
is given by
$
 G = \la J(F) \ra.
$
The lattice $M$ is given by
$$
 M = \lc (i, j, k, l) \in \bZ^4 \relmid 3i + 5 j +10  k +2 l = 0 \rc.
$$
We choose a basis of $M$ as
\begin{align*}
 e_1 &= (-1,-1,1,-1), & e_2 &= (-1,1,0,-1), & e_3 &= (-1,1,-1,4). 
\end{align*}
With respect this basis,
the polytope $\Delta_{(F,G)}$ is the convex hull of
\begin{align*}
\begin{array}{c|c|c|c}
 z^2 & w^{10}& x^5y&y^2z  \\  \hline (1,0,0) & (1,-2,2) & (-2,-1,-1) & (0,1,0)
\end{array}
.
\end{align*}
Similarly,
we choose the basis of the group
$$
 \Mv := \lc (i, j, k, l) \in \bZ^4 \relmid  2i + 4 j +  3k + l = 0 \rc
$$
as
\begin{align*}
\ev_1 &= (-2,0,1,1), & \ev_2 &= (-1,1,0,-2), & \ev_3 &= (-1,0,0,2),
\end{align*}
so that the Newton polytope $\Delta_{(\Fv, \Gv)}$
of $\Fv/xyzw$ is the convex hull of
\begin{align*}
\begin{array}{c|c|c|c}
yz^2 & w^{10}& x^5& xy^2  \\ \hline (1,0,-1) & (-1,-1,4) & (-1,-1,-1) & (-1,1,1)
\end{array}
.
\end{align*}
The polytope
\begin{align*}
 \Delta := \lc (1,0,0),(1,-2,2),(-2,-1,-1),(-1,2,0),(-2,-2,-1) \rc
\end{align*}
contains $\Delta_{(F,G)}$,
and its polar dual polytope
\begin{align*}
 \Delta^\circ := \lc (-1,-1,-1),(-1,1,1),(-1,1,2),(-1,-1,4),(1,0,-1),(-1,0,3) \rc
\end{align*}
contains $\Delta_{(\Fv,\Gv)}$.

\subsection{$U_{1,0}$-singularity}

The invertible polynomial
$
 f = x^3y+y^2z+z^3
$
is weighted homogeneous of degree 
$
 (q_1, q_2, q_3; h) = (2,3,3;9).
$
The group of maximal diagonal symmetries is given by
\[
 \Gmax(f)
  = \la J(f),\, \frac{1}{2}(1,1,0)\ra.
\]
The four-variable extension
$
 F = x^3y+y^2z+z^3+w^9
$
of $f$ is weighted homogeneous
of weight $(2,3,3,1)$ and satisfies \pref{as:F}.
The subgroup $G$ of $\Gmax(F)$
is given by
$
 G = \la J(F), \, \frac{1}{2}(1,1,0,0) \ra.
$
The lattice $M$ is given by
$$
 M = \lc (i, j, k, l) \in \bZ^4 \relmid 2i + 3 j + 3 k + l = 0 \text{ and } i+j \equiv 0 \text{ mod }  2 \rc.
$$
We choose a basis of $M$ as
\begin{align*}
 e_1 &= (1,1,0,-5), & e_2 &= (1,-1,0,1), & e_3 &= (0,0,1,-3). 
\end{align*}
With respect this basis,
the polytope $\Delta_{(F,G)}$ is the convex hull of
\begin{align*}
\begin{array}{c|c|c|c}
 z^3 &w^9 &x^3y &y^2z  \\  \hline (-1,0,2) & (-1,0,-1) & (1,1,-1) & (0,-1,0)
\end{array}
.
\end{align*}
Similarly,
we choose the basis of the group
$$
 \Mv := \lc (i, j, k, l) \in \bZ^4 \relmid  3i + 3 j + 2 k + l = 0 \rc
$$
as
\begin{align*}
\ev_1 &= (1,0,-1,-1), & \ev_2 &= (1,-1,0,0), & \ev_3 &= (-1,0,2,-1),
\end{align*}
so that the Newton polytope $\Delta_{(\Fv, \Gv)}$
of $\Fv/xyzw$ is the convex hull of
\begin{align*}
\begin{array}{c|c|c|c}
x^3 & w^9 & xy^2& yz^3 \\ \hline (1,1,0) & (-5,1,-3) & (1,-1,0) & (0,0,1)
\end{array}
.
\end{align*}
The polytope
\begin{align*}
 \Delta := \lc (-1,0,2),(-1,0,-1),(1,2,-1),(1,1,-1),(0,-1,0),(0,-1,-1) \rc 
\end{align*}
contains $\Delta_{(F,G)}$,
and its polar dual polytope
\begin{align*}
 \Delta^\circ := \lc (1,1,0),(1,-1,0),(0,0,1),(-3,0,-2),(-2,1,0),(-5,1,-3) \rc
\end{align*}
contains $\Delta_{(\Fv,\Gv)}$.

\subsection{$E_{18}$-singularity}

The invertible polynomial
$
 f = x^5 + y^3+xz^2
$
is weighted homogeneous of degree 
$
 (q_1, q_2, q_3; h) = (3,5,6;15).
$
The group of maximal diagonal symmetries is given by
\[
 \Gmax(f)
  = \la J(f), \, \frac{1}{2}(0,0,1) \ra.
\]
The four-variable extension
$
 F = x^5 + y^3+xz^2+zw^9
$
of $f$ is weighted homogeneous
of weight $(3,5 ,6 ,1 )$ and satisfies \pref{as:F}.
The subgroup $G$ of $\Gmax(F)$
is given by
$
 G = \la J(F), \, \frac{1}{2}(0,0,1,1) \ra.
$
The lattice $M$ is given by
$$
 M = \lc (i, j, k, l) \in \bZ^4 \relmid 3i + 5 j + 6 k + l = 0 \text{ and } k+l \equiv 0 \text{ mod 2 }  \rc.
$$
We choose a basis of $M$ as
\begin{align*}
 e_1 &= (1,1,0,-8), & e_2 &= (0,1,1,-11), & e_3 &= (1,0,1,-9).
\end{align*}
With respect this basis,
the polytope $\Delta_{(F,G)}$ is the convex hull of
\begin{align*}
\begin{array}{c|c|c|c}
 xz^2 &w^9z &x^5 &y^3  \\  \hline (-1,0,1) & (-1,0,0) & (2,-3,2) & (1,1,-2)
\end{array}
.
\end{align*}
Similarly,
we choose the basis of the group
$$
 \Mv := \lc (i, j, k, l) \in \bZ^4 \relmid  i + 3 j + 4 k + l = 0 \rc
$$
as
\begin{align*}
\ev_1 &= (-1,1,-1,2), & \ev_2 &= (0,1,0,-3), & \ev_3 &= (0,-2,1,2),
\end{align*}
so that the Newton polytope $\Delta_{(\Fv, \Gv)}$
of $\Fv/xyzw$ is the convex hull of
\begin{align*}
\begin{array}{c|c|c|c}
z^2w &w^9 &x^5z &y^3  \\ \hline (1,2,2) & (1,-2,0) & (-4,-5,-4) & (1,1,0)
\end{array}
.
\end{align*}
The polytope
\begin{align*}
 \Delta := \lc (1,1,-2),(-1,0,0),(-1,0,1),(2,-3,2),(1,-2,1) \rc
\end{align*}
contains $\Delta_{(F,G)}$,
and its polar dual polytope
\begin{align*}
 \Delta^\circ := \lc (1,2,2),(1,-2,0),(-8,-11,-9),(0,1,1),(1,1,0)\rc
\end{align*}
contains $\Delta_{(\Fv,\Gv)}$.

\subsection{$E_{19}$-singularity}

The invertible polynomial
$
 f = x^7+xy^3+z^2
$
is weighted homogeneous of degree 
$
 (q_1, q_2, q_3; h) = (2,4,7;14).
$
The group of maximal diagonal symmetries is given by
\[
 \Gmax(f)
  = \la J(f), \, \frac{1}{3}(0,1,0) \ra.
\]
The four-variable extension
$
 F = x^7+xy^3+z^2+yw^{10}
$
of $f$ is weighted homogeneous
of weight $(2,4 ,7 ,1 )$ and satisfies \pref{as:F}.
The subgroup $G$ of $\Gmax(F)$
is given by
$
 G = \la J(F), \, \frac{1}{3}(0,1,0,-1) \ra.
$
The lattice $M$ is given by
$$
 M = \lc (i, j, k, l) \in \bZ^4 \relmid 2i + 4 j + 7 k + l = 0 \text{ and } j-l \equiv 0 \text{ mod 3}  \rc.
$$
We choose a basis of $M$ as
\begin{align*}
  e_1 &= (3,0,0,-6), & e_2 &= (1,-1,0,2), & e_3 &= (0,1,1,-11). 
\end{align*}
With respect this basis,
the polytope $\Delta_{(F,G)}$ is the convex hull of
\begin{align*}
\begin{array}{c|c|c|c}
 z^2 &w^{10}y &x^7 &xy^3  \\  \hline (-1,2,1) & (0,-1,-1) & (2,0,-1) & (1,-3,-1)
\end{array}
.
\end{align*}
Similarly,
we choose the basis of the group
$$
 \Mv := \lc (i, j, k, l) \in \bZ^4 \relmid  i +  3j + 5 k + l = 0 \rc
$$
as
\begin{align*}
\ev_1 &= (2,1,-1,0), & \ev_2 &= (0,-3,2,-1), & \ev_3 &= (-1,-1,1,-1),
\end{align*}
so that the Newton polytope $\Delta_{(\Fv, \Gv)}$
of $\Fv/xyzw$ is the convex hull of
\begin{align*}
\begin{array}{c|c|c|c}
z^2 &w^{10} &x^7y &y^3w  \\ \hline (0,0,1) & (-6,2,-11) & (3,1,0) & (0,-1,1)
\end{array}
.
\end{align*}
The polytope
\begin{align*}
 \Delta := \lc (-1,2,1),(0,-1,-1),(1,0,-1),(2,0,-1),(1,-3,-1) \rc
\end{align*}
contains $\Delta_{(F,G)}$,
and its polar dual polytope
\begin{align*}
 \Delta^\circ := \lc (0,0,1),(-6,2,-11),(4,2,-1),(1,-1,2),(0,-1,1) \rc
\end{align*}
contains $\Delta_{(\Fv,\Gv)}$.

\subsection{$E_{20}$-singularity}

The invertible polynomial
$
 f =x^{11}+y^3+z^2
$
is weighted homogeneous of degree 
$
 (q_1, q_2, q_3; h) = (6,22,33;66).
$
The group of maximal diagonal symmetries is given by
\[
 \Gmax(f)
  = \la J(f)\ra.
\]
The four-variable extension
$
 F = x^{11}+y^3+z^2+xw^{12}
$
of $f$ is weighted homogeneous
of weight $(6,22 ,33 ,5 )$ and satisfies \pref{as:F}.
The subgroup $G$ of $\Gmax(F)$
is given by
$
 G = \la J(F) \ra.
$
The lattice $M$ is given by
$$
 M = \lc (i, j, k, l) \in \bZ^4 \relmid 6i +  22j +  33k + 5l = 0  \rc.
$$
We choose a basis of $M$ as
\begin{align*}
 e_1 &= (-1,-1,1,-1), & e_2 &= (0,-1,-1,11), & e_3 &= (2,-1,0,2). 
\end{align*}
With respect this basis,
the polytope $\Delta_{(F,G)}$ is the convex hull of
\begin{align*}
\begin{array}{c|c|c|c}
 z^2 &w^{12}x &x^{11} &y^3  \\  \hline (1,0,0) & (0,1,0) & (-2,-1,4) & (-1,0,-1)
\end{array}
.
\end{align*}
Similarly,
we choose the basis of the group
$$
 \Mv := \lc (i, j, k, l) \in \bZ^4 \relmid  i + 4 j + 6 k + l = 0 \rc
$$
as
\begin{align*}
\ev_1 &= (-2,-1,1,0), & \ev_2 &= (-1,0,0,1), & \ev_3 &= (4,-1,0,0),
\end{align*}
so that the Newton polytope $\Delta_{(\Fv, \Gv)}$
of $\Fv/xyzw$ is the convex hull of
\begin{align*}
\begin{array}{c|c|c|c}
 z^2 &w^{12} &wx^{11} &y^3  \\ \hline (1,-1,0) & (-1,11,2) & (-1,0,2) & (-1,-1,-1)
\end{array}
.
\end{align*}
The polytope
\begin{align*}
 \Delta := \lc (1,0,0) , (0,1,0) , (-2,-1,4) , (-1,0,-1) \rc
\end{align*}
contains $\Delta_{(F,G)}$,
and its polar dual polytope
\begin{align*}
 \Delta^\circ := \lc (-1,-1,-1),(-1,-1,2),(1,-1,0),(-1,11,2)\rc
\end{align*}
contains $\Delta_{(\Fv,\Gv)}$.

\subsection{$Z_{17}$-singularity}

The invertible polynomial
$
 f = x^4y+y^3+xz^2 
$
is weighted homogeneous of degree 
$
 (q_1, q_2, q_3; h) = (2,4,5;12).
$
The group of maximal diagonal symmetries is given by
\[
 \Gmax(f)
  = \la J(f) \ra.
\]
The four-variable extension
$
 F = x^4y+y^3+xz^2+zw^7
$
of $f$ is weighted homogeneous
of weight $(2,4,5,1;12)$ and satisfies \pref{as:F}.
The subgroup $G$ of $\Gmax(F)$
is given by
$
 G = \la J(F) \ra.
$
The lattice $M$ is given by
$$
 M = \lc (i, j, k, l) \in \bZ^4 \relmid 2i + 4 j + 5 k + l = 0 \rc.
$$
We choose a basis of $M$ as
\begin{align*}
e_1 &= (1,1,0,-6), & e_2 &= (1,0,1,-7), & e_3 &= (0,1,1,-9). 
\end{align*}
With respect this basis,
the polytope $\Delta_{(F,G)}$ is the convex hull of
\begin{align*}
\begin{array}{c|c|c|c}
 xz^2 &w^7z &x^4y &y^3  \\  \hline (-1,1,0) & (-1,0,0) & (2,1,-2) & (1,-2,1)
\end{array}
.
\end{align*}
Similarly,
we choose the basis of the group
$$
 \Mv := \lc (i, j, k, l) \in \bZ^4 \relmid  i + 2 j + 3 k + l = 0 \rc
$$
as
\begin{align*}
\ev_1 &= (-1,1,-1,2), & \ev_2 &= (0,-2,1,1), & \ev_3 &= (0,1,0,-2), 
\end{align*}
so that the Newton polytope $\Delta_{(\Fv, \Gv)}$
of $\Fv/xyzw$ is the convex hull of
\begin{align*}
\begin{array}{c|c|c|c}
wz^2 & w^7 & x^4z & xy^3 \\ \hline (1,2,2) & (1,0,-2) & (-3,-3,-4) & (0,-1,0)
\end{array}
.
\end{align*}
The polytope
\begin{align*}
 \Delta := \lc (-1,0,0), (1,-2,1), (1,0,-1), (0,1,-1), (2,1,-2), (-1,1,0) \rc
\end{align*}
contains $\Delta_{(F,G)}$,
and its polar dual polytope
\begin{align*}
 \Delta^\circ := \lc (1,2,2), (1,1,2), (0,1,1), (1,0,-2), (1,0,1), (0,-1,0), (-6,-7,-9) \rc
\end{align*}
contains $\Delta_{(\Fv,\Gv)}$.

\subsection{$Z_{19}$-singularity}

The invertible polynomial
$
 f = x^9y+y^3+z^2
$
is weighted homogeneous of degree 
$
 (q_1, q_2, q_3; h) = (4,18,27;54).
$
The group of maximal diagonal symmetries is given by
\[
 \Gmax(f)
  = \la J(f)\ra.
\]
The four-variable extension
$
 F = x^9y+y^3+z^2+xw^{10}
$
of $f$ is weighted homogeneous
of weight $(4,18 ,27 ,5 )$ and satisfies \pref{as:F}.
The subgroup $G$ of $\Gmax(F)$
is given by
$
 G = \la J(F) \ra.
$
The lattice $M$ is given by
$$
 M = \lc (i, j, k, l) \in \bZ^4 \relmid 4i + 18 j + 27 k + 5l = 0 \rc.
$$
We choose a basis of $M$ as
\begin{align*}
 e_1 &= (-1,-1,1,-1), & e_2 &= (-1,2,-1,-1), & e_3 &= (0,-1,-1,9). 
\end{align*}
With respect this basis,
the polytope $\Delta_{(F,G)}$ is the convex hull of
\begin{align*}
\begin{array}{c|c|c|c}
 z^2 &w^{10}x &x^9y &y^3  \\  \hline (1,0,0) & (0,0,1) & (-5,-3,-1) & (0,1,0)
\end{array}
.
\end{align*}
Similarly,
we choose the basis of the group
$$
 \Mv :=  \lc (i, j, k, l) \in \bZ^4 \relmid  i + 3 j + 5 k + l = 0 \rc
$$
as
\begin{align*}
\ev_1 &= (-5,0,1,0), & \ev_2 &= (-3,1,0,0), & \ev_3 &= (-1,0,0,1),
\end{align*}
so that the Newton polytope $\Delta_{(\Fv, \Gv)}$
of $\Fv/xyzw$ is the convex hull of
\begin{align*}
\begin{array}{c|c|c|c}
z^2 & w^{10}&wx^9 & xy^3 \\ \hline (1,-1,-1) & (-1,-1,9) & (-1,-1,0) & (-1,2,-1)
\end{array}
.
\end{align*}
The polytope
\begin{align*}
 \Delta := \lc (1,0,0),(0,1,0),(0,0,1),(-3,-2,0),(-5,-3,-1) \rc
\end{align*}
contains $\Delta_{(F,G)}$,
and its polar dual polytope
\begin{align*}
 \Delta^\circ := \lc (-1,-1,-1),(-1,2,-1),(-1,2,0),(-1,-1,9),(1,-1,-1) \rc
\end{align*}
contains $\Delta_{(\Fv,\Gv)}$.

\subsection{$Q_{17}$-singularity}

The invertible polynomial
$
 f = x^5z+xy^3+z^2
$
is weighted homogeneous of degree 
$
 (q_1, q_2, q_3; h) = (1,3,5;10).
$
The group of maximal diagonal symmetries is given by
\[
 \Gmax(f)
  = \la J(f), \, \frac{1}{3}(0,1,0) \ra.
\]
The four-variable extension
$
 F = x^5z+xy^3+z^2+yw^7
$
of $f$ is weighted homogeneous
of weight $(1,3 ,5 ,1 )$ and satisfies \pref{as:F}.
The subgroup $G$ of $\Gmax(F)$
is given by
$
 G = \la J(F), \, \frac{1}{3}(0,1,0,-1) \ra.
$
The lattice $M$ is given by
$$
 M = \lc (i, j, k, l) \in \bZ^4 \relmid i +  3j + 5 k + l = 0 \text{ and } j-l \equiv 0 \text{ mod 3}  \rc.
$$
We choose a basis of $M$ as
\begin{align*}
  e_1 &= (3,0,0,-3), & e_2 &= (0,1,1,-8), & e_3 &= (1,0,1,-6). 
\end{align*}
With respect this basis,
the polytope $\Delta_{(F,G)}$ is the convex hull of
\begin{align*}
\begin{array}{c|c|c|c}
 z^2 &w^7y &xy^3 &x^5z  \\  \hline (-1,-1,2) & (0,0,-1) & (1,2,-3) & (1,-1,1)
\end{array}
.
\end{align*}
Similarly,
we choose the basis of the group
$$
 \Mv := \lc (i, j, k, l) \in \bZ^4 \relmid  i + 2 j + 3 k + l = 0 \rc
$$
as
\begin{align*}
\ev_1 &= (1,1,-1,0), & \ev_2 &= (-1,2,-1,0), & \ev_3 &= (1,-3,2,-1),
\end{align*}
so that the Newton polytope $\Delta_{(\Fv, \Gv)}$
of $\Fv/xyzw$ is the convex hull of
\begin{align*}
\begin{array}{c|c|c|c}
 xz^2 &w^7 &x^5y &wy^3  \\ \hline (0,1,1) & (-3,-8,-6) & (3,0,1) & (0,1,0)
\end{array}
.
\end{align*}
The polytope
\begin{align*}
 \Delta := \lc (-1,-1,2),(0,0,-1),(0,-1,0),(3,-1,0),(1,2,-3) \rc
\end{align*}
contains $\Delta_{(F,G)}$,
and its polar dual polytope
\begin{align*}
 \Delta^\circ := \lc  (0,1,1),(-3,-8,-6),(4,-1,1),(2,1,1),(0,1,0) \rc
\end{align*}
contains $\Delta_{(\Fv,\Gv)}$.

\subsection{$Q_{18}$-singularity}

The invertible polynomial
$
 f =x^8z+y^3+z^2
$
is weighted homogeneous of degree 
$
 (q_1, q_2, q_3; h) = (3,16,24;48).
$
The group of maximal diagonal symmetries is given by
\[
 \Gmax(f)
  = \la J(f) \ra.
\]
The four-variable extension
$
 F = x^8z+y^3+z^2+xw^9
$
of $f$ is weighted homogeneous
of weight $(3,16,24,5)$ and satisfies \pref{as:F}.
The subgroup $G$ of $\Gmax(F)$
is given by
$
 G = \la J(F) \ra.
$
The lattice $M$ is given by
$$
 M = \lc (i, j, k, l) \in \bZ^4 \relmid 3i + 16 j + 24 k + 5l = 0 \rc.
$$
We choose a basis of $M$ as
\begin{align*}
 e_1 &= (-1,-1,1,-1), & e_2 &= (-1,2,-1,-1), & e_3 &= (0,-1,-1,8). 
\end{align*}
With respect this basis,
the polytope $\Delta_{(F,G)}$ is the convex hull of
\begin{align*}
\begin{array}{c|c|c|c}
 z^2 &w^9x &x^8z &y^3  \\  \hline (1,0,0) & (0,0,1) & (-4,-3,-1) & (0,1,0)
\end{array}
.
\end{align*}
Similarly,
we choose the basis of the group
$$
 \Mv := \lc (i, j, k, l) \in \bZ^4 \relmid  i +  3j + 4 k + l = 0 \rc
$$
as
\begin{align*}
\ev_1 &= (-4,0,1,0), & \ev_2 &= (-3,1,0,0), & \ev_3 &= (-1,0,0,1),
\end{align*}
so that the Newton polytope $\Delta_{(\Fv, \Gv)}$
of $\Fv/xyzw$ is the convex hull of
\begin{align*}
\begin{array}{c|c|c|c}
xz^2 &w^9 &wx^8 &y^3  \\ \hline (1,-1,-1) & (-1,-1,8) & (-1,-1,0) & (-1,2,-1)
\end{array}
.
\end{align*}
The polytope
\begin{align*}
 \Delta := \lc (1,0,0),(0,1,0),(0,0,1),(-3,-2,0),(-4,-3,-1) \rc
\end{align*}
contains $\Delta_{(F,G)}$,
and its polar dual polytope
\begin{align*}
 \Delta^\circ := \lc (-1,-1,-1),(1,-1,-1),(1,-1,0),(-1,-1,8),(-1,2,-1) \rc
\end{align*}
contains $\Delta_{(\Fv,\Gv)}$.

\subsection{$W_{18}$-singularity}

The invertible polynomial
$
 f = x^7+y^2z+z^2
$
is weighted homogeneous of degree 
$
 (q_1, q_2, q_3; h) = (4,7,14;28).
$
The group of maximal diagonal symmetries is given by
\[
 \Gmax(f)
  = \la J(f) \ra.
\]
The four-variable extension
$
 F = x^7+y^2z+z^2+xw^8
$
of $f$ is weighted homogeneous
of weight $(4,7,14,3)$ and satisfies \pref{as:F}.
The subgroup $G$ of $\Gmax(F)$
is given by
$
 G = \la J(F) \ra.
$
The lattice $M$ is given by
$$
 M = \lc (i, j, k, l) \in \bZ^4 \relmid 4i + 7 j +14  k +3 l = 0  \rc.
$$
We choose a basis of $M$ as
\begin{align*}
 e_1 &= (0,2,-1,0), & e_2 &= (1,1,-1,1), & e_3 &= (-1,0,-1,6). 
\end{align*}
With respect this basis,
the polytope $\Delta_{(F,G)}$ is the convex hull of
\begin{align*}
\begin{array}{c|c|c|c}
 z^2 &w^8x &x^7 &y^2z  \\  \hline (0,-1,0) & (-1,1,1) & (-3,5,-1) & (1,-1,0)
\end{array}
.
\end{align*}
Similarly,
we choose the basis of the group
$$
 \Mv := \lc (i, j, k, l) \in \bZ^4 \relmid  i + 4 j + 2 k + l = 0 \rc
$$
as
\begin{align*}
\ev_1 &= (-3,1,0,-1), & \ev_2 &= (5,-1,-1,1), & \ev_3 &= (-1,0,0,1),
\end{align*}
so that the Newton polytope $\Delta_{(\Fv, \Gv)}$
of $\Fv/xyzw$ is the convex hull of
\begin{align*}
\begin{array}{c|c|c|c}
yz^2 &w^8 &wx^7 &y^2  \\ \hline (-1,-1,-1) & (0,1,6) & (0,1,-1) & (2,1,0)
\end{array}
.
\end{align*}
The polytope
\begin{align*}
 \Delta := \lc (0,-1,0),(-1,1,1),(-3,5,-1),(2,-1,0),(0,0,1) \rc
\end{align*}
contains $\Delta_{(F,G)}$,
and its polar dual polytope
\begin{align*}
 \Delta^\circ := \lc (-1,-1,-1),(0,1,6),(0,1,-1),(1,1,-1),(2,1,0) \rc
\end{align*}
contains $\Delta_{(\Fv,\Gv)}$.

\subsection{$S_{17}$-singularity}

The invertible polynomial
$
 f = x^6y+y^2z+z^2
$
is weighted homogeneous of degree 
$
 (q_1, q_2, q_3; h) = (1,2,4;8).
$
The group of maximal diagonal symmetries is given by
\[
 \Gmax(f)
  = \la J(f), \, \frac{1}{3}(1,0,0)\ra.
\]
The four-variable extension
$
 F = x^6y+y^2z+z^2+xw^7
$
of $f$ is weighted homogeneous
of weight $(1,2,4,1)$ and satisfies \pref{as:F}.
The subgroup $G$ of $\Gmax(F)$
is given by
$
 G = \la J(F), \, \frac{1}{3}(1,0,0,-1) \ra.
$
The lattice $M$ is given by
$$
 M =  \lc (i, j, k, l) \in \bZ^4 \relmid i +  2j + 4 k + l = 0 \text{ and } i-l \equiv 0 \text{ mod 3}  \rc.
$$
We choose a basis of $M$ as
\begin{align*}
 e_1 &= (1,-1,0,1), & e_2 &= (0,1,1,-6), & e_3 &= (0,0,3,-12). 
\end{align*}
With respect this basis,
the polytope $\Delta_{(F,G)}$ is the convex hull of
\begin{align*}
\begin{array}{c|c|c|c}
 z^2 &w^7x &x^6y &y^2z  \\  \hline (-1,-2,1) & (0,-1,0) & (5,5,-2) & (-1,0,0)
\end{array}
.
\end{align*}
Similarly,
we choose the basis of the group
$$
 \Mv := \lc (i, j, k, l) \in \bZ^4 \relmid  i + 3 j + 2 k + l = 0 \rc
$$
as
\begin{align*}
\ev_1 &= (5,-1,-1,0), & \ev_2 &= (5,0,-2,-1), & \ev_3 &= (-2,0,1,0),
\end{align*}
so that the Newton polytope $\Delta_{(\Fv, \Gv)}$
of $\Fv/xyzw$ is the convex hull of
\begin{align*}
\begin{array}{c|c|c|c}
yz^2 &w^7 &wx^6 &xy^2  \\ \hline (0,1,3) & (1,-6,-12) & (1,0,0) & (-1,1,0)
\end{array}
.
\end{align*}
The polytope
\begin{align*}
 \Delta := \lc (-1,-2,1),(0,-1,0),(6,5,-2),(5,5,-2),(-1,2,-1) \rc
\end{align*}
contains $\Delta_{(F,G)}$,
and its polar dual polytope
\begin{align*}
 \Delta^\circ := \lc (0,1,3),(1,-6,-12),(1,1,2),(-1,1,0),(0,-3,-7) \rc
\end{align*}
contains $\Delta_{(\Fv,\Gv)}$.
\subsection{$U_{16}$-singularity}

The invertible polynomial
$
 f = x^5+y^2z+yz^2
$
is weighted homogeneous of degree 
$
 (q_1, q_2, q_3; h) = (3,5,5;15).
$
The group of maximal diagonal symmetries is given by
\[
 \Gmax(f)
  = \la J(f)\ra.
\]
The four-variable extension
$
 F = x^5+y^2z+yz^2+xw^6
$
of $f$ is weighted homogeneous
of weight $(3,5,5,2)$ and satisfies \pref{as:F}.
The subgroup $G$ of $\Gmax(F)$
is given by
$
 G =  \la J(F) \ra.
$
The lattice $M$ is given by
$$
 M =  \lc (i, j, k, l) \in \bZ^4 \relmid 3i + 5 j +5  k +2 l = 0 \rc.
$$
We choose a basis of $M$ as
\begin{align*}
 e_1 &= (-1,-1,0,4), & e_2 &= (-1,0,-1,4), & e_3 &= (0,-1,-1,5). 
\end{align*}
With respect this basis,
the polytope $\Delta_{(F,G)}$ is the convex hull of
\begin{align*}
\begin{array}{c|c|c|c}
 yz^2 &w^6x &x^5 &y^2z  \\  \hline (1,0,-1) & (0,0,1) & (-2,-2,3) & (0,1,-1)
\end{array}
.
\end{align*}
Similarly,
we choose the basis of the group
$$
 \Mv := \lc (i, j, k, l) \in \bZ^4 \relmid  i + 2 j + 2 k + l = 0 \rc
$$
as
\begin{align*}
\ev_1 &= (-2,0,1,0), & \ev_2 &= (-2,1,0,0), & \ev_3 &= (3,-1,-1,1),
\end{align*}
so that the Newton polytope $\Delta_{(\Fv, \Gv)}$
of $\Fv/xyzw$ is the convex hull of
\begin{align*}
\begin{array}{c|c|c|c}
yz^2 &w^6 &wx^5 &y^2z  \\ \hline (0,-1,-1) & (4,4,5) & (-1,-1,0) & (-1,0,-1)
\end{array}
.
\end{align*}
The polytope
\begin{align*}
 \Delta := \lc (1,0,-1) , (0,0,1) , (-2,-2,3) , (0,1,-1) \rc
\end{align*}
contains $\Delta_{(F,G)}$,
and its polar dual polytope
\begin{align*}
 \Delta^\circ := \lc (1,-2,-1),(4,4,5),(-2,-2,-1),(-2,1,-1) \rc
\end{align*}
contains $\Delta_{(\Fv,\Gv)}$.
\bibliographystyle{amsalpha}
\bibliography{bibs}

\noindent
Makiko Mase

Department of Mathematics and Information Sciences,
Tokyo Metropolitan University,
1-1 Minami-Osawa,
Hachioji-shi
Tokyo,
192-0397,
Japan.

Advanced Mathematical Institute,
Osaka City University,
3-3-138 Sugimoto,
Sumiyoshi-ku,
Osaka,
558-8585,
Japan.

{\em e-mail address}\ : \  mtmase@arion.ocn.ne.jp
\ \vspace{0mm} \\

\noindent
Kazushi Ueda

Department of Mathematics,
Graduate School of Science,
Osaka University,
Machikaneyama 1-1,
Toyonaka,
Osaka,
560-0043,
Japan.

{\em e-mail address}\ : \  kazushi@math.sci.osaka-u.ac.jp
\ \vspace{0mm} \\

\end{document}